\documentclass[11pt]{amsart}
\usepackage{geometry}                
\usepackage{graphicx}
\usepackage{amssymb}
\usepackage{epstopdf}
\usepackage{amsmath}
\usepackage{amsfonts}
\usepackage{amssymb}
\usepackage{amsthm}
\usepackage{color}
\usepackage[parfill]{parskip}
\usepackage{etoolbox}
\usepackage{scrextend}

\usepackage{xcolor}
\usepackage{color}

\usepackage[utf8]{inputenc} 
\usepackage[T1]{fontenc}
\usepackage[all]{xy}

\newtheorem{theorem}{Theorem}[section]
\newtheorem{corollary}[theorem]{Corollary}
\newtheorem{lemma}[theorem]{Lemma}
\newtheorem{prop}[theorem]{Proposition}

\newtheorem{definition}[theorem]{Definition}

\newtheorem{example}[theorem]{Example}
\newtheorem{remark}[theorem]{Remark}

\newtheorem*{theorem*}{Theorem}

\newcommand{\op}{\overline{\mathbb P}_n}
\newcommand{\up}{\underline{\mathbb P}_n}

\newcommand{\st}{{\rm {St}}}
\newcommand{\St}{{\rm {St}}}
\newcommand{\lk}{{\rm {Lk}}}
\newcommand{\su}{\subseteq}
\newcommand{\pa}{\partial}
\newcommand{\Z}{{\mathbb Z}}
\renewcommand{\i}{{\hat i}}
\newcommand{\E}{{\mathcal E}}
\newcommand{\M}{{\mathcal M}}
\newcommand{\U}{{\mathcal U}}
\newcommand{\R}{\Bbb R}

\newcommand{\N}{{\mathbb N}}
\newcommand{\cc}{{\frak c}}
\newcommand{\tc}{{\sf TC}}

\newtheorem{lema}{Lemma}
\newtheorem{teo}[lema]{Theorem}

\theoremstyle{remark}

\theoremstyle{definition}

\newtheorem{ej}[lema]{Example}

\title[Large simplicial complexes]{Large simplicial complexes: \\  Universality, Randomness, and Ampleness}
\author{Michael Farber}
\date{\today}                                           

\begin{document}

\begin{abstract}{The paper surveys recent progress in understanding geometric, topological and combinatorial properties of large simplicial complexes, focusing mainly on ampleness, connectivity and universality  \cite{EZFM}, \cite{FM}, \cite{Rado}. 
In the first part of the paper we concentrate on $r$-ample 
simplicial complexes which are high dimensional analogues of the $r$-e.c. graphs introduced originally 
by Erd\H os and R\'eniy \cite{ER}, see also \cite{Bonato}. The class of $r$-ample complexes is useful for applications since these complexes allow extensions of subcomplexes of certain type in all possible ways; besides, $r$-ample complexes exhibit remarkable robustness properties.
We discuss results about the existence of $r$-ample complexes and describe their probabilistic and deterministic constructions. The properties of random simplicial complexes in medial regime \cite{FM} 
are important for this discussion since these complexes are ample, in certain range. We prove that the topological complexity of a random simplicial complex in the medial regime satisfies $\tc(X)\le 4$, with probability tending to $1$ as $n\to\infty$. 
There exists a unique (up to isomorphism) $\infty$-ample complex on countable set of vertexes (the Rado complex), and the second part of the paper surveys the results about universality, homogeneity, indestructibility and other important properties of this complex.
 The Appendix written by J.A. Barmak discusses connectivity of conic and ample complexes. 
}\end{abstract}

\maketitle
\section{Introduction}

Network science uses large simplicial complexes for modelling complex networks consisting of an enormous number of interacting objects. The pairwise interactions can be modelled by a graph, but the higher order interactions between the objects require the language of simplicial complexes, see \cite{Bat}.

 In this survey article we discuss $r$-ample simplicial complexes representing {\lq\lq stable and resilient\rq\rq} networks, in the sense that small alterations of the network have limited impact on its global properties (such as connectivity and high connectivity). We also discuss a remarkable simplicial complex $X$ ({\it the Rado complex}) which is 
{ \lq\lq totally  indestructible\rq\rq } in the following sense: 
removing any finite number of simplexes of $X$ leaves a simplicial complex isomorphic to $X$. 
The complex $X$ has infinite (countable) number of vertexes and cannot be practically implemented. 
The $r$-ample simplicial complexes can be viewed as finite approximations to the Rado complex, 
they retain a limited degree of 
indestructibility. The formal definition of $r$-ampleness requires the existence of all possible extensions of simplicial subcomplexes of  size at most $r$.

A related mathematical object is {\it the medial regime random simplicial complex} \cite{FM}, which is $r$-ample with probability tending to one. Informally, the Rado complex can be viewed as a limit of the random simplicial complex in the medial regime. The geometric realisation of the Rado complex is homeomorphic to the infinite dimensional simplex and hence it is contractible.
It was proven in  \cite{FM} that the medial regime random simplicial complex is simply connected and 
has vanishing Betti numbers in dimensions below $\log_2 \log_2 n.$ For these reason one expects that any $r$-ample simplicial complexes is highly connected, for large $r$. 
This question is discussed below in \S \ref{sec:6} and in the Appendix. 
 
Analogues of the ampleness property of \cite{EZFM} have been studied in literature for graphs, hypergraphs, tournaments, and other structures, in combinatorics and in mathematical logic, and
a variety of terms have been used: 
$r$-existentially completeness, $r$-existentially closeness, {$r$-e.c.}~for short,  (see \cite{Che}, \cite{Bonato})
and also the {Adjacency Axiom $r$} (see \cite{BH, BH1}),  an {extension property} \cite{Fa}, 
{property $P(r)$} \cite{Bo}. This property intuitively means that you can get anything you want, for this reason it is also referred to as the 
{Alice's Restaurant Axiom} \cite{Spe}, \cite{Win}.

The main theme of this paper is universality and its relation to randomness, in the realm of simplicial complexes. Speaking about universality one should certainly mention the Urysohn metric space $\U$, a remarkable mathematical object constructed 
by P.S. Urysohn in 1920's. 
The space $\U$ is  universal in the sense that it contains an isometric copy of any complete, separable metric space. Additionally, the Urysohn space $\U$ is homogeneous in the sense that any partial isometry between its finite subsets can be extended to a global isometry. The properties of universality and homogeneity determine $\U$ uniquely up to isometry. A.M. Vershik \cite{V} defines the notion of a random metric space and proves that such space with probability 1 is isomorphic to the Urysohn universal metric space.  

{\it The Rado graph} $\Gamma$ is another notable mathematical  object, which can also be characterised by its universality and homogeneity, see \cite{C}, \cite{C1}. 
The graph $\Gamma$ has countably many vertexes, and it is universal in the sense that any graph with countably many vertexes is isomorphic to an induced subgraph of $\Gamma$. Moreover,  any isomorphism between finite induced subgraphs of $\Gamma$ can be extended to the whole of $\Gamma$ (homogeneity). The properties of universality and homogeneity determine $\Gamma$ uniquely up to isomorphism. 
Erd\H os and R\'enyi \cite{ER} showed that a random graph on countably many vertexes 
is isomorphic to $\Gamma$ with probability 1; this result explains why $\Gamma$ is  sometimes called {\it \lq\lq the random graph\rq\rq}. 
Rado \cite{Ra} suggested a deterministic construction of $\Gamma$ in which the vertexes $V(\Gamma)$ are 
labelled by integers $\N$ and a pair of vertexes labelled by $m<n$ are connected by an edge iff the $m$-th digit in the binary expansion of $n$ is $1$. This same graph construction implicitly appeared in an earlier paper by W. Ackermann \cite{Ack}, who studied the consistence of the axioms of set theory. 


The Rado simplicial complex $X$ introduced in \cite{Rado} can also be characterised by universality and homogeneity and we also know that a random simplicial complex on countably many vertexes (in a certain regime) is isomorphic to $X$ with probability 1. 
One observes several curious properties of $X$, for example if the set of vertexes of $X$ is partitioned into finitely many parts, the simplicial complex induced on at least one of these parts is isomorphic to $X$. The link of any simplex of $X$ is isomorphic to $X$. 
One of the key properties of $X$ is its {\it indescructibility}: removing any finite set of simplexes leaves a simplicial complex isomorphic to $X$.

The main source for the present survey are the papers \cite{EZFM}, \cite{FM}, \cite{Rado}. Next we comment on the other related publications. 

Theorem 3  of R. Rado \cite{Ra} suggests a construction of a universal uniform hypergraph of a fixed dimension $\ell$. 
Equivalently, uniform hypergraphs can be understood as simplicial complexes of a fixed dimension 
$\ell$ having complete $(\ell-1)$-dimensional skeleta. 

In \cite{BH} Blass and Harary study the 0-1 law for the first order language of simplicial complexes
of fixed dimension $\ell$ with respect to the counting probability measure. They show that a typical 
$\ell$-dimensional simplicial complex has a full $(\ell-1)$-skeleton. In \cite{BH}, the authors introduce \lq\lq Axiom $n$\rq\rq,  which 
generalises the characteristic property of the Rado graph;  
it is a special case of our notion of ampleness.

The preprint \cite{BTT} of A. Brooke-Taylor and D. Tesla
applies the methods of mathematical logic and model theory to study the geometry of simplicial complexes. A well-known general construction of model theory is {\it the Fra\" \i ss\'e limit} for a class of relational structures possessing certain amalgamation properties, see \cite{WH}. 
The Fra\" \i ss\'e limit construction, when applied to the class of all finite simplicial complexes, produces a simplicial complex 
$F$ on countably many vertexes which is universal and homogeneous, i.e. it is a Rado complex. 
Therefore, the approach of \cite{BTT} offers an interesting different viewpoint on the Rado complex. 
In \cite{BTT} the authors study the group of automorphisms of $F$ and state that any direct limit of finite groups and any metrisable profinite group embeds into the group of automorphisms of $F$. Besides, \cite{BTT} contains a proof that the geometric realisation of $F$ is homeomorphic to an infinite-dimensional simplex.
The authors of \cite{BTT} also consider a probabilistic approach and claim the 0-1 law for first order theories.

\section{Ample simplicial complexes}

We use the following notations. The symbol $V(X)$ denotes the set of vertices of a simplicial complex $X$. 
If $U\subseteq V(X)$ is a subset we denote by $X_U$ the {\it induced subcomplex} on $U$, i.e., $V(X_U)=U$ and a set of points of $U$ forms a simplex in $X_U$ if and only if it is a simplex in $X$. 

An {\it embedding} of a simplicial complex $A$ into $X$ is an isomorphism between $A$ and an induced subcomplex of $X$.

 The \emph{join} of simplicial complexes $X$ and $Y$ is denoted $X \ast Y$; recall that the set of vertexes of the join is $V(X)\sqcup V(Y)$ and the simplexes of the join are simplexes of the complexes $X$ and $Y$ as well as simplexes of the form
 $\sigma\tau=\sigma\ast\tau$ where $\sigma$ is a simplex in $X$ and $\tau$ is a simplex in $Y$. 
 The simplex $\sigma\tau=\sigma\ast\tau$ has as its vertex set the union of vertexes of $\sigma$ and of $\tau$. 
 The symbol $CX$ stands for the \emph{cone} over $X$. For a vertex $v\in V(X)$ the symbol $\lk_X(v)$ denotes the {\it link} of $v$ in $X$, i.e., the subcomplex of $X$ formed by all simplexes $\sigma$ which do not contain $v$ but such that $\sigma v=\sigma\ast v$ is a simplex of $X$. 
 
 Besides, the symbol $F(X)$ denotes the set of all simplexes of $X$ and $E(X)$ denotes the set of all external simplexes of $X$, i.e. such that $\sigma\notin X$ and $\partial \sigma\subset X$. 

The following definition was introduced in \cite{EZFM}:

\begin{definition}
\label{first}
\label{def:ample}
Let $r\ge 1$ be an integer. A nonempty simplicial complex $X$ is said to be {\it $r$-ample} if for each subset $U\subseteq V(X)$ with $|U|\le r$ and for each subcomplex $A\subseteq X_U$ there exists a vertex $v\in V(X) - U$ such that 
\begin{eqnarray}\label{11a}\label{link}
{\lk}_X(v)\cap X_U \;=\; A.
\end{eqnarray} 
We say that $X$ is {\it ample or $\infty$-ample} if it is $r$-ample for every $r\ge 1$. 
\end{definition}

The condition (\ref{11a}) 
can equivalently be expressed as 
$X_{U\cup \{v\}} \;=\; X_U \cup (v\ast A). $
This is illustrated by Figure \ref{fig:xua}. 
\begin{figure}[h]
    \centering
    \includegraphics[scale = 0.35]{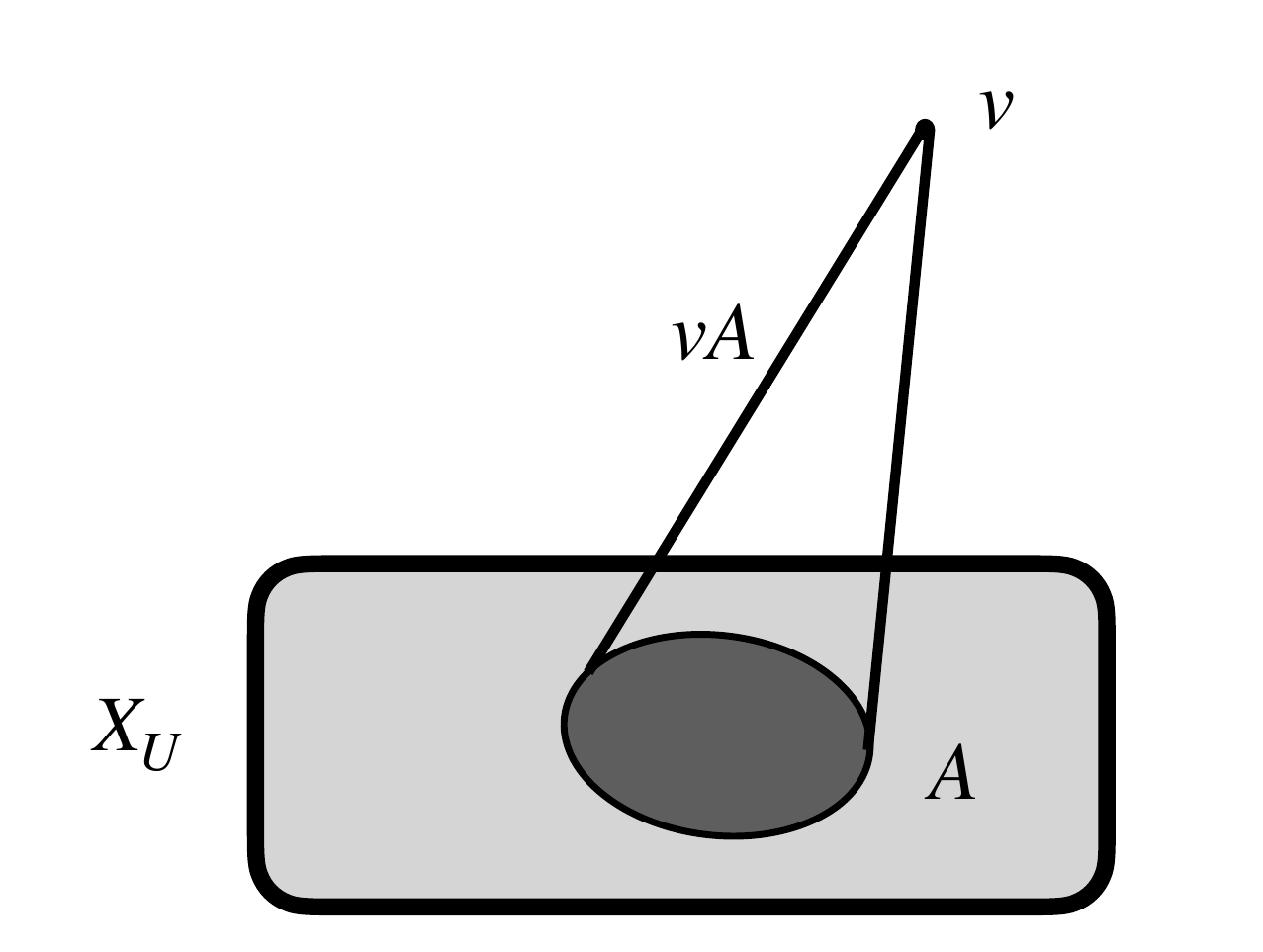}
    \caption{The complex $X_U\cup (vA)$.}
    \label{fig:xua}
\end{figure}

It is clear that $r$-ampleness depends only on the $r$-dimensional skeleton.

Here is an alternative characterisation of $r$-ampleness, see Lemma 2.3 in \cite{EZFM}:

\begin{lemma}\label{lm:equiv}
\label{AB}
A simplicial complex $X$ is $r$-ample if and only if for every pair $(A,B)$ consisting of a simplicial complex $A$ and an induced subcomplex $B$ of ~$A$, satisfying $|V(A)|\le r+1$, and for every embedding $f_B$ of $B$ into $X$, there exists an embedding $f_A$ of $A$ into $X$ extending $f_B$. 
\end{lemma}

A 2-ample complex is obviously connected, and the example below shows that a 2-ample complex may not be simply connected. 

\begin{example}{\rm 
\label{thirteen} 
Consider a 2-dimensional simplicial complex $X$ having 13 vertexes labelled by integers $0, 1, 2, \dots, 12$. 
A pair of vertexes $i$ and $j$ is connected by an edge iff the difference $i-j$ is a square modulo 13, i.e. if
$i-j \equiv \pm 1, \pm 3, \pm 4 \mod 13.$ The 1-skeleton of $X$ is a well-known Paley graph of order 13. 
Next we add 13 triangles $i, i+1, i+4, $ where  $i=0, 1, \dots, 12.$
We claim that the obtained complex $X$ is 2-ample. 
The verification amounts to the following: for any two vertices, there exists other ones adjacent to both, to neither, only to one, and only to the other. 
Moreover, any edge lies both on a single filled and unfilled triangles. Indeed, an edge $i, i+1$ lies in the triangle $i, i+1, i+4$ (filled) as well as in the triangle $i-3, i, i+1$ (unfilled). Informally, the filled triangles can be characterised by the identity $1+3=4$ and the unfilled by $3+1=4$. 

We note that $X$ can be obtained from the triangulated torus with 13 vertexes, 39 edges and 26 triangles by removing 13 white triangles of type 
$i, i+3, i+4$. 
From this description it is obvious that $X$ collapses onto a graph and calculating the Euler characteristic we find $b_0(X)=1,$ $b_1(X)= 14$ and $b_2(X)=0$. Thus, we see that $X$ is not simply connected. 
}
\end{example}

\begin{remark}{\rm 
The link of a vertex in an $r$-ample simplicial complex is $(r-1)$-ample. More generally, 
 the link of every $k$-dimensional simplex in an $r$-ample complex is $(r-k-1)$-ample. }
 \end{remark}

\begin{example}\label{exbar}{\rm 
J.A. Barmak \cite{Bar} constructed for every $n\ge 0$ an infinite $(2n+1)$-ample simplicial complex which is not $n$-connected. In particular, this shows that a $3$-ample complex can be not simply connected. It would be interesting to have a finite example of this kind. 

The construction of \cite{Bar} starts with the complex $K_0=S^0\ast S^0\ast \dots \ast S^0$, the join of $n+1$ copies of $S^0$. Clearly $K_0$ has $2n+2$ vertices and is homeomorphic to the $n$-dimensional sphere $S^n$. For each $i\ge 0$, the complex $K_{i+1}$ is obtained from $K_i$ by attaching a cone over every subcomplex $L\subset K$ with at most $2n+1$ vertexes. The complex $K=\cup_{i\ge 0} K_i$ is 
obviously $(2n+1)$-ample. The proof that $K$ is not $n$-connected is based on considering
 the fundamental class $[K_0]\in H_n(K_0)$ and observing that its image under the homomorphism $H_n(K_0)\to H_n(K)$ is nonzero. Details can be found in \cite{Bar}, Theorem 1.

}
\end{example}

From Lemma \ref{AB} it follows that an $r$-ample complex $X$ satisfies: (a) any simplicial complex on at most $r+1$ vertexes can be embedded into $X$ and (b) $\dim X\ge r$. 

An $r$-ample complex must be fairly large. To make this statement precise we shall denote by $M'(r)$ the number of simplicial complexes with vertexes from the set $\{1, 2, \dots, r\}$. 
The number $M'(r)+1=M(r)$ is known as {\it the Dedekind number}, see \cite{kleitman1975}, it equals the number of monotone Boolean functions of $r$ variables and has some other combinatorial interpretations, for example, it equals the number of antichains in the set of $r$ elements. 
A few first values of {\it \lq\lq the reduced Dedekind number\rq\rq} $M'(r)$ are $M'(1)=2$, $M'(2)=5$, $M'(3)=19$. For  general $r$,  
the number $M'(r)$ admits estimates
\begin{eqnarray}\label{dede}
{\binom r {\lfloor r/2 \rfloor}}\,  \le\,  \log_2 (M'(r))\, \le \, {\binom r {\lfloor r/2 \rfloor} \left(1+O\left(\frac{\log r}{r}\right)\right)}.
\end{eqnarray}
The lower bound in (\ref{dede}) is easy: one counts only the simplicial complexes having the full $\lfloor n/2 \rfloor$ skeleton; the upper bound in (\ref{dede}) has been obtained in \cite{kleitman1975}. 
Using the Stirling formula one obtains 
\begin{eqnarray}\label{frml25}
\log_2 \log_2 (M'(r)) = r - \frac{1}{2} \log_2 r +O(1).
\end{eqnarray}
\begin{corollary}\label{cor26}\label{lower}
An $r$-ample simplicial complex contains at least $$M'(r)+r \ge 2^{\binom r {\lfloor r/2\rfloor}} +r$$ 
vertexes. 
\end{corollary}
\begin{proof} Let $X$ be an $r$-ample complex.
We can embed into $X$ an $(r-1)$-dimensional simplex $\Delta$ having $r$ vertexes. Applying Definition \ref{first},
for every subcomplex $A$ of $\Delta$ we can find a vertex $v_A$ in the complement of $\Delta$ having $A$ as its link intersected with $\Delta$. The number of subcomplexes $A$ is $M'(r)$ and we also have $r$ vertexes of $\Delta$ which gives the estimate. 
\end{proof}

\section{Existence of ample complexes}\label{exist}

\begin{theorem}\label{thm31}
For every $r\ge 5$ and for every $n\ge r2^r2^{2^r}$ there exists an $r$-ample simplicial complex having exactly $n$ vertexes. 
\end{theorem}

This was proven in \S 5 of \cite{EZFM} using the probabilistic method. We briefly indicate below the main steps of the proof. 

Consider a random subcomplex $X$ of the standard simplex $\Delta_n$ on the vertex set 
$\{1, 2, \dots, n\}$ with the following probability function: the probability of a simplicial complex $X$ equals 
\begin{eqnarray}\label{gibbs}
\mathbb P(X)=2^{-H(X)}, \quad\mbox{where}\quad H(X)=|F(X)|+|E(X)|,\end{eqnarray}
is the sum of the total number $|F(X)|$ of simplexes of $X$ and the number $|E(X)|$ of external simplexes of $X$. Recall that an external simplex $\sigma$ is such that $\sigma\notin X$ but $\partial \sigma\subset X$. 
Formula (\ref{gibbs}) is a special case of the medial regime assumption, compare formula (3) in \cite{FM}, which reduces to (\ref{gibbs}) when 
\begin{eqnarray}\label{1/2}
p_\sigma= 1/2\end{eqnarray}
for all simplexes $\sigma\in \Delta_n$. 
In other words, it is a special case of the multi-parameter model of random simplicial complexes when each simplex is selected with probability $1/2$. 
The arguments of the proof of Proposition 5.1 in \cite{EZFM} show that the probability that a random subcomplex $X\subset \Delta_n$ is not 
$r$-ample is bounded above by 
\begin{eqnarray}\label{prob}
 n^r\cdot 2^{2^r}(1-2^{-2^r})^{n-r}
\end{eqnarray}
One can show that for $n\ge r2^r2^{2^r}$ the number (\ref{prob}) is smaller than one implying the existence of $r$-ample simplicial complexes. 

\begin{theorem} For any fixed $r\ge 1$, a random simplicial complex on $n$ vertexes with respect to the measure (\ref{gibbs}) is $r$-ample with probability tending to $1$ as $n\to \infty$. 
\end{theorem}

This is a consequence of the estimate (\ref{prob}) since for any fixed $r$ the expression (\ref{prob}) tends to $0$ when $n\to \infty$. 


\section{Paley type construction of ample complexes}

This section briefly describes an explicit construction of ample complexes \cite{EZFM} which generalises the well-known 
construction of Paley graphs. 

Fix an odd prime power $n$, an odd prime~$p$ that divides $n-1$, and a primitive element~$g$ in the finite field~$\mathbb{F}_n$. The subset $ Q_{n,p}\subset \mathbb{F}_n$ is defined as follows
$$ Q_{n,p} =\; \left\{g^{\alpha} \;\mid\; \alpha \equiv \beta^2 \bmod p, \;\text{for}\; \alpha,\beta \in \mathbb{Z} \right\} \;\subset\; \mathbb{F}_n.$$
Note that $H = \langle g^p \rangle\subset \mathbb{F}_n^{\times} = \langle g \rangle$ is a multiplicative subgroup of index $p$, since $p|(n-1)$, 
and there is a group isomorphism $\mathbb{F}_n^{\times}/H \to (\mathbb{F}_p,+)$ taking $g H \mapsto 1$. The set $Q_{n,p}$ is the union of $H$-cosets that correspond to quadratic residues mod~$p$, it contains about half of the elements of the field, more precisely
$$ \left|Q_{n,p}\right| \;=\; \frac{p+1}{2p}\,(n-1).$$


%

\begin{definition}
\label{Xnp}
The Iterated Paley Simplicial Complex $X_{n,p}$ has $\mathbb{F}_n$ as its vertex set and a non-empty subset $\{x_1, x_2, \dots, x_t\}\subset \mathbb{F}_n$ forms a simplex if for every subset \newline 
$\{x_{s_1}, x_{s_2}, \dots, x_{s_k}\} \subseteq \{x_1, x_2, \dots, x_t\}$ one has
$$ \prod_{1\le i<j\le k}\left(x_{s_i}-x_{s_j}\right) \;\in\; Q_{n,p}.$$
\end{definition}

Note that $(-1)=g^{(n-1)/2}$, and $(n-1)/2\equiv 0 \bmod p$ since $p$ is odd, hence $(-1)\in H = \langle g^p \rangle$. Therefore, the condition in the definition of~$X_{n,p}$ does not depend on the order of the vertices $x_1, x_2, \dots, x_t$. Note also that all $n$ singletons~$\{x\}$ are hyperedges, because $1 = g^0 \in Q_{n,p}$.

The definitions of $Q_{n,p}$ and hence  of $X_{n,p}$ depend on the choice of primitive element $g \in \mathbb{F}_n$. Any other primitive element $h = g^\alpha \in \mathbb{F}_n$ gives the same construction if $\alpha$ is a quadratic residue mod~$p$, and a different one if not. The two constructions are not isomorphic in general. The result below applies to either choice.

%

\begin{theorem}
\label{amplepaley} 
Let $r \in \mathbb{N}$. Every Iterated Paley Simplicial Complex $X_{n,p}$ with $p>2^{2^r+2r}$ and $n>r^2p^{2r}$ is $r$-ample.
\end{theorem}

We refer to \cite{EZFM} for the proof and further remarks.

Note that the probabilistic construction of 
Theorem \ref{thm31} generate
ample complexes with a smaller number of vertexes compared to Theorem \ref{amplepaley}.

\section{Resilience of ample complexes}
\label{3}

\S 3 of \cite{EZFM} contains results characterising \, \lq\lq resilience\rq\rq\, of $r$-ample simplicial complexes: small perturbations to the complex
reduce its ampleness in a controlled way and hence many important geometric properties pertain. 

The perturbations we have in mind are as follows. 
If $X$ is a simplicial complex and $\mathcal F$ is a finite set of simplexes of $X$, one may consider the simplicial complex $Y$ obtained from $X$ by removing all simplexes of $\mathcal F$ as well as all simplexes which have faces 
belonging to $\mathcal F$. We shall say that $Y$ {\it is obtained from $X$ by removing the set of simplexes $\mathcal F$}. 


We shall characterise the size of $\mathcal F$ by two numbers: $|\mathcal F|$ (the cardinality of $\mathcal F$) and 
$ \dim(\mathcal F)=\sum_{\sigma\in \mathcal F} \dim \sigma$ ({\it \lq\lq the total dimension\rq\rq\ of $\mathcal F$}). 
\begin{theorem}\label{remove1}
Let $X$ be an $r$-ample simplicial complex and let $Y$ be obtained from $X$ by removing a set  $\mathcal F$ of simplexes. Then $Y$ is $(r-k)$-ample provided that 
 \begin{eqnarray}\label{less}
  |\mathcal F|+\dim(\mathcal F) <  M'(k)+k.
 \end{eqnarray}
 In particular, taking into account (\ref{dede}), the complex $Y$ is $(r-k)$-ample if 
  \begin{eqnarray}\label{less1}
|\mathcal F|+\dim(\mathcal F) < 2^{\binom k {\lfloor k/2\rfloor}}+k . 
 \end{eqnarray}
 \end{theorem}
 
To illustrate this general result suppose that $X$ is $r$-ample where $r\ge 3$ and $Y$ is obtained from $X$ by 
removing a set of $a_0$ vertexes and $a_1$ edges. The complex $Y$ will be connected provided it is $2$-ample. Applying Theorem \ref{remove1} with $k=r-2$ we see that the inequality 
\begin{eqnarray}\label{con1}
a_0+2a_1<M'(r-2)+r-2
\end{eqnarray}
guaranties the $2$-ampleness and hence the connectivity of $Y$. 
The following more explicit inequality
\begin{eqnarray}\label{con2}
a_0+2a_1<2^{\binom{r-2}{\lfloor r/2\rfloor -1}}+r-2.
\end{eqnarray}
implies (\ref{con1}) as follows from (\ref{dede}).  

\begin{proof}[Proof of Theorem \ref{remove1}] Without loss of generality we may assume that $\mathcal F$ forms an anti-chain, i.e. 
no simplex of $\mathcal F$ is a proper face of another simplex of $\mathcal F$. 
Indeed, if $\sigma_1\subset \sigma_2$, where $\sigma_1, \sigma_2\in \mathcal F$,
we can remove $\sigma_2$ from $\mathcal F$ without affecting the complex $Y$.


Consider a vertex $v\in V(Y)$ and its links $\lk_Y(v)\subset \lk_X(v)$ in $Y$ and in $X$, correspondingly. 
Denote by $\mathcal F_v$ the set of simplexes $\sigma\subset \lk_X(v)$ such that either $\sigma\in \mathcal F$ or $v\sigma\in \mathcal F$. As follows directly from the definitions,  {\it $\lk_Y(v)$ is obtained from $\lk_X(v)$ by removing the set of simplexes $\mathcal F_v$}.

Represent $\mathcal F$ as the disjoint union 
$$\mathcal F \, =\, \mathcal F_0\sqcup \mathcal F_{1},$$
 where $\mathcal F_0$ is the set of zero-dimensional simplexes from $\mathcal F$ and $\mathcal F_{1}$ is the set of simplexes in $\mathcal F$ having dimension $\ge 1$. 
 Denote by 
 $$W_0=\cup_{\sigma\in \mathcal F_0} V(\sigma) \quad \mbox{and}\quad W_1=\cup_{\sigma\in \mathcal F_1}V(\sigma)$$ 
 the sets of zero-dimensional simplexes and the set of vertexes of simplexes of positive dimension in $\mathcal F$. 
 Note that $W_0\cap W_1=\emptyset$ due to our anti-chain assumption. Besides, $V(Y)=V(X)-W_0$ and therefore 
 $W_1\subset V(Y)$. 

Let  $U\subset V(Y)$ be a subset and let $v\in V(Y)$ be a vertex such that: \newline (a) $v\notin W_1$ and (b) 
the set $\lk_X(v)\cap X_U$ is a subcomplex of $Y_U$. Then 
\begin{eqnarray}\label{conclusion}
\lk_Y(v)\cap Y_U\, =\, \lk_X(v)\cap X_U. 
\end{eqnarray}
Indeed, we have $\lk_X(v)\cap Y_U =\lk_Y(v)\cap Y_U$ because of our assumption (a) and  
$\lk_X(v)\cap Y_U = \lk_X(v)\cap X_U$ because of (b).

Let $k$ be an integer satisfying (\ref{less}) and 
let $U\subset V(Y)$ be a subset with $|U|\le r-k$. 
Given a subcomplex $A\subset Y_{U}$, we want to show the existence of a vertex $v\in V(Y)-U$ such that 
\begin{eqnarray}\label{inters}
\lk_Y(v)\cap Y_{U} =A.\end{eqnarray}
This would mean that our complex $Y$ is $(r-k)$-ample.

Consider the induced subcomplex $X_{U}$ which obviously contains $A$ as a subcomplex. Consider also the abstract simplicial complex 
$$K=X_{U}\cup (A\ast \Delta),$$ 
where $\Delta$ is an abstract full simplex on $k$ vertexes. 
Note that 
$K$ has at most $r$ vertexes, 
$X_{U}$ is an induced subcomplex of $K$ and it is naturally embedded into $X$. 
Using the assumption that $X$ is $r$-ample and applying Lemma \ref{AB}, we can find an embedding of $K$ into $X$ extending the identity map of $X_{U}$. 
In other words, we can find $k$ vertexes $v_1, \dots, v_k\in V(X)-U$ such that for a simplex $\tau$ of $X_{U}$ 
and for any subset $$\tau'\subset \{v_1, \dots, v_k\}=U'$$ one has 
$\tau\tau' \in X$ if and only if $\tau \in A$. 
If one of these vertexes $v_i$ lies in $V(Y)-W_1$ then (using (\ref{conclusion}))
$$\lk_Y(v_i)\cap Y_{U}=\lk_X(v_i)\cap X_{U}=A$$ 
and we are done. Thus, without loss of generality, we can assume that 
\begin{eqnarray}\label{includ}
U'\subset W_0\cup W_1.
\end{eqnarray}

Let $Z\subset \Delta$ be an arbitrary simplicial subcomplex. 
We may use the $r$-ampleness of $X$ and apply Definition \ref{first} to the subcomplex $A\sqcup Z$ of 
$X_{U\cup U'}$. 
This gives a vertex $v_Z\in V(X)-(U\cup U')$ satisfying 
$$\lk_{X}(v_Z) \cap X_{U\cup U'}  =A\sqcup Z$$ and in particular, 
\begin{eqnarray}\label{inters2}
\lk_{X}(v_Z) \cap X_{U}  =A.\end{eqnarray}
For distinct subcomplexes $Z, Z'\subset \Delta$ the points $v_Z$ and $v_{Z'}$ are distinct and the cardinality of the set 
$\{v_Z; Z\subset \Delta\}$ equals $M'(k)$. 
Noting that (\ref{inters2}) is a subcomplex of $Y_U\subset X_U$ and 
comparing (\ref{conclusion}), (\ref{inters}), (\ref{inters2}), we see that our statement would follow once we know that 
the vertex $v_Z$ lies in $V(Y) -W_1$ at least for one subcomplex $Z$. 

Let us assume the contrary, i.e. 
$v_Z \in (W_0\cup W_1) - U'$
for every subcomplex $Z\subset \Delta$. The cardinality of the set $\{v_Z\}$ equals $M'(k)$ and the cardinality of the set $(W_0\cup W_1) - U'$ equals $|\mathcal F|+\dim \mathcal F - k$ and we get a contradiction with our assumption 
(\ref{less}). 

This completes the proof. 
\end{proof}

\section{Connectivity of ample complexes} \label{sec:6}

We observed above that 2-ample complex is connected and 3-ample complex may be not simply connected. However 
a 4-ample complex must be simply connected:

\begin{prop}\label{lem1cycle}
For $r\ge 4$, any $r$-ample simplicial complex $Y$ is simply connected. 
Moreover, any simplical loop $\alpha: S^1\to Y$ 
with $n$ vertexes in an $r$-ample complex $Y$ bounds a simplicial disc $\beta: D^2\to Y$ where $D^2$ is a triangulation of the disc having $n$ boundary vertexes, at most  
$\lceil\frac{n-3}{r-3}\rceil $ internal vertexes and at most $\lceil \frac{n-3}{r-3}\rceil \cdot(r-1)+1$ triangles. 
\end{prop}
\begin{proof}
If $n\le r$ we may simply apply the definition of $r$-ampleness and find an extension $\beta: D^2\to Y$ with a single internal vertex. If $n>r$ we may apply the definition of $r$-ampleness to any arc consisting of $r$ vertexes, see Figure
\ref{fig:gamma}. This reduces the length of the loop by $r-3$ and performing $\lceil\frac{n-r}{r-3}\rceil$ such operations we obtain a loop of length $\le r$ which can be filled by a single vertex. The number of internal vertexes of the bounding disc will be $\lceil\frac{n-r}{r-3}\rceil +1= \lceil\frac{n-3}{r-3}\rceil.$
\begin{figure}[h]
    \centering
    \includegraphics[scale = 0.6]{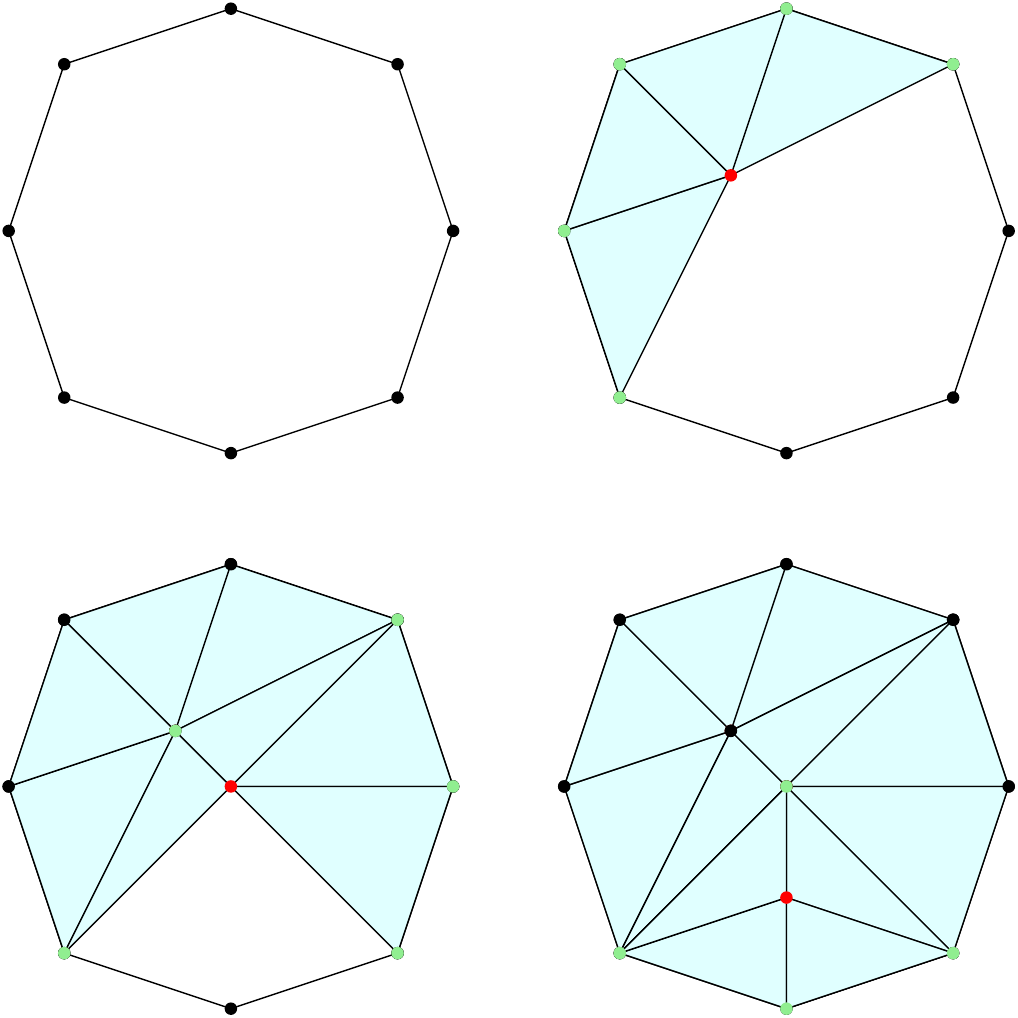}
    \caption{The process of constructing the bounding disc in a $5$-ample complex}
    \label{fig:gamma}
\end{figure}
To estimate the number of triangles we note that on each intermediate step of the process described above we add $r-1$ triangles and on the final step we may add at most $r$ triangles. This leads to the upper bound $\lceil \frac{n-r}{r-3}\rceil \cdot (r-1) +r= 
\lceil \frac{n-3}{r-3}\rceil \cdot(r-1)+1$. 
\end{proof}

Next we state a general result about connectivity of ample complexes:

\begin{theorem}\label{connected} An $r$-ample complex is $(\lfloor\frac{r}{2}\rfloor -1)$-connected. 
\end{theorem}

This follows from Theorem \ref{thmbarmak} proven by J. A. Barmak, see the Appendix. 

In paper \cite{Bar} J.A. Barmak  introduced the notion of conic simplicial complex.
The class of conic complexes includes the class of ample complexes and is more convenient for studying questions about connectivity. 

\begin{definition}
For an integer $r\ge 0$, a simplicial complex $K$ is said to be {\it $r$-conic} if every subcomplex $L\subset K$ 
with at most $r$ vertexes is contained the closed star $\st_K(v)$ of a vertex $v\in K$. 
\end{definition}
Note that the notions of $0$-conicity and $1$-conicity are equivalent to the complex $K$ to be non-empty. 

\begin{theorem}\label{thmbarmak} 
{\rm [J.A. Barmak]} For $r\ge 2$ any $r$-conic simplicial complex is $(\lfloor r/2\rfloor -1)$-connected. 
\end{theorem}


\begin{example}{\rm Consider the $(r+1)$-fold simplicial join $S^0\ast S^0\ast \dots \ast S^0$ which is 
homeomorphic to the sphere $S^{r}$ and is obviously $(2r+1)$-conic. 
This example shows that in general $(2r+1)$-conicity does not imply $r$-connectivity. In other words, the statement of Theorem \ref{thmbarmak} is sharp. 
}
\end{example}

\section{Random simplicial complexes in the medial regime}

In \S \ref{exist} we briefly mentioned a special class of random simplicial complexes in the medial regime which were studied in  \cite{FM}. These are random subcomplexes of the standard simplex $\Delta_n$ with probability measure (\ref{gibbs}). We are interested in asymptotic properties of these complexes as $n\to \infty$. A geometric or topological property of random simplicial complexes is satisfied asymptotically almost surely (a.a.s.) if the probability that it holds tends to $1$ as $n\to \infty$. 

We emphasise that the measure (\ref{gibbs}) is a special case of the multi-parameter probability measure studied in \cite{costa}, \cite{costa1}, \cite{costa2}, \cite{costa3}, where one sets as probability parameters $p_\sigma=1/2$ for all simplexes $\sigma$. 

The main results of \cite{FM} under the assumptions (\ref{1/2}) can be summarised as follows.
We shall use the notation $\beta(n) = \log_2\log_2 n + \log_2\log_2\ln n$. 

\begin{theorem}\label{fm}
Fix arbitrary $\epsilon_0>0$ and $\delta_0>0$. Then: 

(a) The dimension of a random complex $X$ in the medial regime satisfies 
$$\lfloor \beta(n)\rfloor -1\le \dim X \le \beta(n) -1+\epsilon_0,$$
a.a.s.

(b) A random complex $X$ is connected and simply connected, a.a.s.

(c) The Betti numbers $b_j(X)$ vanish for all \, $0<j \le \log_2\log_2 n -1 -\delta_0$, a.a.s. 
\end{theorem}

One may strengthen the above statement using Theorem \ref{connected}:

\begin{theorem}\label{thm:con}
Let $r=r(n)$ be an integer valued function satisfying 
\begin{eqnarray}\label{r}
r\le \log_2\log_2 n -\epsilon,\quad \mbox{where}\quad \epsilon >0.\end{eqnarray} 

Then: 

(1) The random complex $X$ in the medial regime is $r$-ample, a.a.s. 

(2) In particular, $X$ is $(\lfloor r/2\rfloor -1)$-connected, a.a.s.
\end{theorem}

To prove the first statement one observes that the expression (\ref{prob}) under the assumption (\ref{r}) tends to 0 as $n\to \infty$. Indeed,  
the logarithm with base $e$ of (\ref{prob}) is bounded above by
\begin{eqnarray}\label{4}
r\cdot \ln n +\ln 2 \cdot 2^r  -n\left(1-\frac{r}{n}\right)\cdot 2^{-2^r}.\end{eqnarray}
Using (\ref{r}) we have 
$n\cdot 2^{-2^r} \ge n^{1-2^{-\epsilon}},$ and since $r/n\to 0$ we have
\begin{eqnarray}\label{comp}
n\left(1-\frac{r}{n}\right)\cdot 2^{-2^r}\ge \frac{1}{2} \cdot n^{1-2^{-\epsilon}}.
\end{eqnarray}
On the other hand, 
\begin{eqnarray}\label{comp1}
r \cdot \ln n +\ln 2\cdot  2^r &\le& 
r\cdot \ln 2\cdot\log_2 n + \ln 2\cdot 2^{-\epsilon}\cdot \log_2 n \nonumber\\
&\le & \ln 2\cdot \log_2 n\cdot (\log_2\log_2 n -\epsilon +2^{-\epsilon}) 
\end{eqnarray}
Comparing (\ref{comp}) with (\ref{comp1}) we see that (\ref{4}) tends to $-\infty$ hence implying that the probability 
(\ref{prob}) of a random complex being not $r$-ample tends to $0$. This proves (1). 

Statement (2) follows from (1) by applying Theorem \ref{connected}. 

\begin{remark}\label{rk:37}
{\rm 
We see that the medial regime random simplicial complex is \newline $(\lfloor \frac{1}{2} \log_2\log_2 n\rfloor -1)$-connected, which is roughly half of the dimensions where its Betti numbers vanish, see Theorem \ref{fm}, (c). 
This leaves open the important question of whether the integral homology groups $H_j(X)$ may be nontrivial (and hence finite) for dimensions $j$ in the interval
$\frac{1}{2}\log_2\log_2 n \le j < \log_2\log_2 n -1.$ }
\end{remark}

\section{Topological complexity of random simplicial complexes in the medial regime}

First we recall the notion of topological complexity $\tc(X)$ (see \cite{F03}, \cite{Finv}), where $X$ is a path-connected topological space. Intuitively, the integer 
$\tc(X)$ is a measure of navigational complexity of $X$ viewed as the configuration space of a system. To give the precise definition, 
consider the path space $X^I$, i.e. the space of all continuous 
maps $I=[0,1] \to X$ equipped with compact-open topology, and the fibration 
\begin{eqnarray}\label{fibr}
\pi : X^I \to X \times X, \quad \alpha \mapsto (\alpha(0), \alpha(1)).
\end{eqnarray} 
The {\it topological complexity} $\tc(X)$ of $X$ is defined as the sectional category of fibration (\ref{fibr}). In other words, $\tc(X)$ is the smallest integer $k\ge 0$ such that there exists and open cover
$$X\times X = U_0\cup U_1\cup \dots\cup U_k$$
of cardinality $k$ with the property that each open set $U_i$ admits an open section of the fibration (\ref{fibr}), 
cf. \cite{F03}. 

A section of fibration (\ref{fibr}) can be viewed as {\it a robot motion planning algorithm}, and the topological complexity
$\tc(X)$ describes singularities of such algorithms \cite{Finv}. 

For information about recent developments related to the notion of $\tc(X)$ we refer the reader to \cite{GraLV}. 

As an illustrative example consider the space $F(\R^d, n)$, the configuration space of $n$ pairwise distinct points in $\R^d$ which was analysed in \cite{Finv}, \S 4.7. 
The space $F(\R^d, n)$ models motion of $n$ robots in $\R^d$ avoiding collisions. The topological complexity of this space is given by
$$\tc(F(\R^d, n))=\left\{
\begin{array}{lll}
2n-2 & \mbox{for}& d \, \, \mbox{odd},\\
2n-3 & \mbox{for}& d \, \, \mbox{even}.
\end{array}
\right.
$$
We use here the normalised version of the topological complexity which is smaller by 1 compered with the notion of \cite{F03}, \cite{Finv}. 

The above example shows that the topological complexity can be arbitrarily large. 

In this section we shall study the situation when the simplicial complex $X$ is random. 
More specifically, we shall assume that $X$ is a random subcomplex of the standard simplex 
$\Delta_n$ on the vertex set $\{1, 2, \dots, n\}$
generated by the medial regime model (\ref{gibbs}). Surprisingly, under these assumptions for large $n\to \infty$ 
the topological complexity $\tc(X)$ is small with probability tending to $1$:

\begin{theorem}\label{thm81} Let $X\subset \Delta_n$ be a random simplicial complex in the medial regime. Then, with probability tending to $1$ as $n\to \infty$, one has $$\tc(X)\le 4.$$ 
\end{theorem}
\begin{proof} By Theorem 4.16 from \cite{Finv} the topological complexity of an $r$-connected simplicial complex $X$ satisfies the following inequality
\begin{eqnarray}\label{upper}
\tc(X)< \frac{2\dim X +1}{r+1}.\end{eqnarray}
If $X\subset \Delta_n$ is medial regime random complex then, by Theorem \ref{fm}, 
$$\dim X \le \log_2\log_2 n + \log_2\log_2\ln n$$
with probability tending to $1$ as $n\to \infty$. On the other hand, by Theorem \ref{thm:con} the random complex $X$ is $r$-connected, where $r=\frac{1}{2} \log_2\log_2n -3,$
with probability tending to $1$ as $n\to \infty$. Thus,  
$$
\frac{2\dim X +1}{r+1}\le 4\cdot \frac{\log_2\log_2 n + \log_2\log_2\ln n}{\log_2\log_2 n -4}\le 4.5
$$
for large $n$, and we obtain from (\ref{upper}) that $\tc(X)\le 4,$ asymptotically almost surely. 
\end{proof} 

\begin{remark}{\rm 
I believe that statement of Theorem \ref{thm81} can potentially be strengthened to state that $\tc(X)=2$ with probability tending to $1$ for a random simplicial complex $X$ in the medial regime. Achieving this will require improving the connectivity threshold of Theorem \ref{thm:con}, see
 Remark \ref{rk:37}.
}\end{remark}

\section{The $\infty$-ample Rado complex}

In this section we shall follow \cite{Rado} and consider simplicial complexes which are $r$-ample for any $r\ge 1$; we call them $\infty$-ample. 

\begin{theorem}\label{thmrado} (a) There exists an $\infty$-ample complex having a countable set of vertexes. 
(b) Any two $\infty$-ample complexes with countable sets of vertexes are isomorphic. 
\end{theorem}

The simplicial complex of Theorem \ref{thmrado} is called {\it the Rado complex}, in honour of Richard Rado who invented the Rado graph. The Rado graph is the 1-dimensional skeleton of the Rado complex $R$, see \cite{C}. 

\begin{proof}[Proof of Theorem \ref{thmrado} (a)]
To prove (a) we shall construct the required complex $X$ as follows. 
Let $X_0$ be a single point and let each complex $X_{n+1}$ (where $n\ge 0$) be obtained from $X_n$ by first adding a finite set of vertexes $v(A)$ labelled by all subcomplexes $A\subset X_n$ (including $A=\emptyset$); then we 
consider the cone $v(A)\ast A$ with apex $v(A)$ and base $A$ and attach each such cone to $X_n$ along the base $A$. Thus, 
$$X_{n+1} =X_n \cup \bigcup_A(v(A)\ast A),
$$
and we have the infinite chain of finite subcomplexes $X_0\subset X_1\subset X_2\subset \dots$. The complex $$X=\cup_{n\ge 1} X_n$$ is $\infty$-ample. Indeed, any finite set of vertexes $U\subset V(X)$ is contained in $V(X_n)$ for some $n$. The induced subcomplex $X_U$ coincides with $(X_n)_U$ and then for any subcomplex $A\subset X_U$ the vertex
 $v= v(A)$ validates the ampleness property of Definition \ref{def:ample}.   \end{proof}
 
 In the proof of (b) we shall use Lemma \ref{lm16} stated below.

 \begin{lemma}\label{lm16} Let $X$ be an $\infty$-ample complex and let $L'\subset L$ be a pair consisting of a finite simplical complex $L$ and an  induced subcomplex $L'$. Let $f': L' \to X_{U'}$ be an isomorphism of simplicial complexes, where $U'\subset V(X)$ is a finite subset. Then there exists a finite subset $U\subset V(X)$ containing $U'$ and an isomorphism $f: L \to X_{U}$ with $f|L' =f'$. 
\end{lemma}

\begin{proof}[Proof of Lemma \ref{lm16}]
It is enough to prove this statement under an additional assumption that $L$ has a single additional vertex, i.e. $|V(L)| - |V(L')| =1$. In this case 
$L$ is obtained from $L'$ by attaching a cone $wA$ where $w\in V(L) - V(L')$ denotes the new vertex and $A\subset L'$ is a subcomplex (the base of the cone). Applying the defining property of the ample complex to the subset $U'\subset V(X)$ and the subcomplex $f'(A)\subset X_{U'}$
we find a vertex $v\in V(X)-U'$ such that $\lk_X(v)\cap X_{U'} =f(A)$. We can set $U=U'\cup \{v\}$ and extend $f'$ to the isomorphism 
$f: L\to X_{U}$ by setting $f(w)=v$.
\end{proof}

\begin{proof}[Proof of Theorem \ref{thmrado} (b)] The proof uses the well-known back-and-forth argument. 
Let $X$ and $X'$ be two $\infty$-ample complexes. Enumerate their vertexes $V(X)=\{v_1, v_2, \dots \}$ and
$V(X') = \{v'_1, v'_2, \dots \}$ and set $U_1=\{v_1\}$ and $U'_1=\{v'_1\}$. The isomorphism $f_1:X_{U_1}\to X'_{U'_1}$ given by $f_1(v_1)=v'_1$. 

Next we define sequences of finite subsets $U_1\subset U_2\subset \dots \subset V(X)$ and 
$U'_1\subset U'_2\subset \dots \subset V(X')$ satisfying $\cup U_n=V(X)$ and $\cup U'_n=V(X')$ and isomorphisms 
$f_n: X_{U_n}\to X'_{U'_n}$ with $f_n|_{X_{U_{n-1}}}=f_{n-1}$. The whole collection $\{f_n\}$ then defines an isomorphism $f: X\to X'$. 

Acting by induction we shall assume that the sets $U_i$ and $U'_i$ and the isomorphisms $f_i$ for all $i\le n$ have been constructed. If $n$ is odd, we shall find the smallest index $i$ such that $v_i\notin U_{n}$ and set 
$U_{n+1}=U_n\cup\{v_i\}$; then applying Lemma \ref{lm16} we can find a vertex $v'_j\in V(X')-U'_n$ and an isomorphism
$f_{n+1}: X_{U_{n+1}}\to X'_{U'_{n+1}}$ extending $f_n$, where $U'_{n+1}=U'_n\cup \{v'_j\}$. 

If $n$ is even we shall find the smallest index $r$ such that $v'_r\notin U'_n$ and set $U'_{n+1}=U'_n\cup \{v_r\}$ and then by Lemma \ref{lm16} we can find a vertex $v_s\in V(X)-U_n$ and an isomorphism 
$f_{n+1}: X_{U_{n+1}}\to X'_{U'_{n+1}}$, extending $f_n$, where $U_{n+1}=U_n\cup \{v_s\}$. 
\end{proof}

\begin{theorem} \label{thm73} (a) The Rado complex $R$ is universal in the sense that every countable simplicial complex is isomorphic to an induced subcomplex of $R$. (b) The Rado complex $R$ is homogeneous in the sense that for every two finite induced subcomplexes $R_U, R_{U'}\subset R$ and for every isomorphism $f:R_U\to R_{U'}$ there exists an isomorphism $F:R\to R$ with $F|R_U=f$. 
(c) Every universal and homogeneous countable simplicial complex is isomorphic to $R$. 
\end{theorem}
\begin{proof} (a) Let $X$ be a simplicial complex with the vertex set $V(X)=\{v_1, v_2, \dots\}$. Using the 
$\infty$-ampleness property of $R$ we can subsequently find a sequence of vertexes $W= \{w_1, w_2, \dots\}\subset V(R)$ and a sequence of isomorphisms $f_n: X_{U_n}\to R_{W_n}$, where $U_n=\{v_1, \dots, v_n\}$ and $W_n=\{w_1, \dots, w_n\}$, such that $f_n$ extends $f_{n-1}$. This gives an isomorphism 
between $X$ and the induced subcomplex $R_W$.

The proof of (b) uses arguments similar to the ones of the proof of Lemma \ref{lm16}. 

(c) Suppose $X$ is universal and homogeneous. Let $U\subset V(X)$ be a finite subset and let $A\subset X_U$ be a subcomplex of the induced complex. Consider an abstract simplicial complex $L=X_U\cup wA$ which obtained from $X_U$ by adding a cone $wA$ with vertex $w$ and base $A$ where $X_U\cap wA= A$. Clearly, $V(L) = U\cup\{w\}$ and by universality, we may find a subset $U'\subset V(X)$ and an isomorphism $g:L \to X_{U'}$. Denoting $w_1=g(w)$, $A_1=g(A)$ and $U_1=g(U)$ we have 
$X_{U'} = X_{U_1} \cup w_1 A_1.$
Obviously, $g$ restricts to an isomorphism $g| X_U\, :\,  X_U \to X_{U_1}$. 
By the homogeneity property we can find an isomorphism $F:X\to X$ with $F|X_U =g|X_U$. Denoting $v=F^{-1}(w_1)$ we shall have 
$X_{U\cup \{v\}} = X_U \cup vA$ as required. Hence, $X$ is $r$-ample for any $r\ge 0$. 
\end{proof}
\section{Indestructibility of the Rado complex}

The main result of this section is Corollary \ref{cor15} illustrating \lq\lq indestructibility or resilience\rq\rq\,  of the Rado simplicial complex. 

\begin{lemma}\label{lm31}
Let $X$ be a Rado complex, let $U\subset V(X)$ be a finite set and let $A\subset X_U$ be a subcomplex. 
Let $Z_{U,A}\subset V(X)$ denote the set of vertexes $v\in V(X)-U$ satisfying 
$\lk_X(v)\cap X_U=A$. 
Then the set $Z_{U,A}$ is infinite and the induced complex on $Z_{U,A}$ is also a Rado complex. 
\end{lemma}
\begin{proof} Consider a finite set $\{v_1, \dots, v_N\}\subset Z_{U, A}$ of such vertexes. 
One may apply the ampleness property to the set 
$U_1=U\cup \{v_1, \dots, v_N\}$ and to the subcomplex $A\subset X_{U_1}$ to find another vertex $v_{N+1}$ satisfying (\ref{link}), i.e. $v_{N+1}\in Z_{U,A}$. This shows that $Z_{U,A}$ must be infinite. 

Let $Y\subset X$ denote the subcomplex induced by $Z_{U, A}$. 
Consider a finite subset $U'\subset Z_{U,A}=V(Y)$ and a subcomplex $A'\subset X_{U'}=Y_{U'}$. Applying 
the ampleness property
to the set $W= U\cup U'\subset V(X)$ and to the subcomplex $A\sqcup A'$ we find a vertex $z\in V(X)-W$ such that 
\begin{eqnarray}\label{links2}
\lk_X(z) \cap X_W = A\cup A'.\end{eqnarray} 
Since $X_{W} \supset X_U\cup X_{U'}$, the equation (\ref{links2}) implies $\lk_X(z)\cap X_U=A$, i.e. $z\in Z_{U, A}$. 
Intersecting both sides of (\ref{links2}) with $X_{U'}=Y_{U'}$ and using $\lk_Y(z) =\lk_X(z)\cap Y$ (since $Y$ is an induced subcomplex) we obtain
$\lk_Y(z) \cap Y_{U'}= A'$
implying that $Y$ is Rado. 
\end{proof}

\begin{corollary}\label{cor15}
Let $X$ be a Rado complex and let $Y$ be obtained from from $X$ by selecting a finite number of simplexes $F\subset F(X)$ 
and deleting all simplexes $\sigma\in F(X)$ which contain simplexes from $F$ as their faces. Then $Y$ is also a Rado complex. 
\end{corollary}
\begin{proof}
Let $U\subset V(Y)$ be a finite subset and let $A\subset Y_U$ be a subcomplex. We may also view $U$ as a subset of $V(X)$ and then 
$A$ becomes a subcomplex of $X_U$ since $ Y_U\subset X_U$. The set of vertexes $v\in V(X)$ satisfying $\lk_X(v)\cap X_U=A$ is infinite (by Lemma \ref{lm31}) and thus we may find a vertex $v\in V(X)$ which is not incident to simplexes from the family $F$. Then $\lk_Y(v)=\lk_X(v)\cap Y$ and we obtain 
 $\lk_Y(v)\cap Y_U=A$. \end{proof}

\begin{corollary}\label{lm166} Let $X$ be a Rado complex. If the vertex set $V(X)$ is partitioned into a finite number of parts then the induced subcomplex on at least one of these parts is a Rado complex. 
\end{corollary}
\begin{proof} It is enough to prove the statement for partitions into two parts. 
Let $V(X)= V_1\sqcup V_2$ be a partition; denote by $X^1$ and $X^2$ the subcomplexes induced by $X$ on $V_1$ and $V_2$ correspondingly. 
Suppose that none of the  subcomplexes $X^{1}$ and $X^{2}$ is Rado. 
Then for each $i=1, 2$ there exists a finite subset $U_i\subset V_i$ and a subcomplex $A_i\subset X^i_{U_i}$ such that no vertex $v\in V_i$ satisfies 
$\lk_{X^i}(v) \cap X^i_{U_i}= A_i$. 
Consider the subset $U= U_1\sqcup U_2\subset V(X)$ and a subcomplex 
$A=A_1\sqcup A_2\subset X_{U}$. Since $X$ is Rado we may find a vertex $v\in V(X)$ with $\lk_X\cap X_U = A. $ Then $v$ lies in $V_1$ or $V_2$ and we obtain a contradiction, since $\lk_{X^i}(v) \cap X^i_{U_i} =A_i.$
\end{proof}

\begin{lemma}
In a Rado complex $X$, the link of every simplex is a Rado complex. 
\end{lemma}
\begin{proof}
Let $Y=\lk_X(\sigma)$ be the link of a simplex $\sigma\in X$. To show that $Y$ is Rado, let $U\subset V(Y)$ be a subset and let $A\subset Y_U$ be a subcomplex. 
We may apply the defining property of the Rado complex (i.e. ampleness) to the subset $U'=U\cup V(\sigma)\subset V(X)$ and to the subcomplex 
$A\sqcup \bar\sigma\subset X_{U'}$; here $\bar \sigma$ denotes the subcomplex containing the simplex $\sigma$ and all its faces.
We obtain a vertex $w\in V(X)-U'$ with $\lk_X(w)\cap X_{U'} = A\sqcup \bar \sigma$ or equivalently, $X_{U'\cup w} = X_{U'} \cup wA$. Note that $w\in Y=\lk_X(\sigma)$ since the simplex $w\sigma$ is in $X$. Besides, $Y_{U\cup w} =Y_U \cup wA$. Hence we see that the link $Y$ is also a Rado complex. 
\end{proof}

\section{The Rado complex $R$ is \lq\lq random\rq\rq.}

In this section we argue that {\it \lq\lq a typical simplicial complex with countable set of vertexes is isomorphic to $R$\rq\rq}. We give two formal justifications of this statement. Firstly, we show that the space of simplicial complexes isomorphic to $R$ is residual in the space of all simplicial complexes with countably many vertexes, i.e. its complement is a countable union of nowhere dense sets. Secondly, we equip the set of countable simplicial complexes with a probability measure and show that the set of simplicial complexes isomorphic to $R$ has measure $1$. 

Let $\Delta_\N$ denote the simplicial complex with the vertex set $\N =\{1, 2, \dots, \}$ and with all finite nonempty subsets of $\N$ as its simplexes. We shall consider the set $\Omega$ of all simplicial subcomplexes  
$X\subset \Delta_\N$. We shall view $\Omega$ as the set of all countable simplicial complexes. 

One can introduce a metric on $\Omega$ making 
it a compact metric space. 
For a non-negative integer $n\ge 1$ the simplex $\Delta_n$ with the vertex set 
$\{1, 2, \dots, n\}\subset \N$ is a subcomplex $\Delta_n\subset \Delta_\N$. Let $\Omega_n$ denote the set of all simplicial subcomplexes of $\Delta_n$. For $X\in \Omega$ let $X_n$ denote the finite simplicial complex $X_n=X\cap \Delta_n$. For $X, Y\in \Omega$ define $h(X, Y)=\max\{n; X_n = Y_n\}$. Then 
\begin{eqnarray}
d(X,Y) =\exp(-h(X,Y))
\end{eqnarray}
is a metric on $\Omega$ (satisfying ultrametric triangle inequality). The topology determined by this metric coincides with the topology of the inverse limit 
$$\Omega = \lim_{\leftarrow}\Omega_n,$$
where each $\Omega_n$ is equipped with the discrete topology. Since $\Omega_n$ is a finite set (hence it is compact), we see that $\Omega$ is compact and is a Baire space. $\Omega$ is homeomorphic to the Cantor set. 

We shall denote by $\mathcal R\subset \Omega$ the set of simplicial complexes isomorphic to the Rado complex. Complexes $X\in \mathcal R$ can be characterised either by the $\infty$-ampleness or by the properties 
universality and homogeneity described in Theorem \ref{thm73}. 

\begin{theorem}
The set $\mathcal R\subset \Omega$ is residual and therefore it is dense in $\Omega$. 
\end{theorem}

\begin{proof}
Let $U_n\subset \Omega$ be the set of all simplicial complexes $X\in \Omega$ such that for every subcomplex $A\subset X_n$ there exists a vertex $v_A\in V(X)-\{1, 2, \dots, n\}$ with the property
\begin{eqnarray}\label{vertex}
\lk_X(v_A)\cap X_n = A.
\end{eqnarray} 
Recall that $X_n$ denotes $X\cap \Delta_n$. 
Clearly, $\mathcal R = \cap_{n\ge 1} U_n,$ and the theorem will follow once we show that each $U_n$ is open and dense in $\Omega$. 

To show that $U_n$ is open, let us assume that $X\in U_n$. Consider the set of all vertexes $v_A$ corresponding (as in (\ref{vertex})) to all subcomplexes $A\subset X_n$. It is a finite set and we may find $m>n$ such that all these vertexes $v_A$ lie in $X_m$. The set $\{Y\in \Omega; Y_m=X_m\}$ is contained in $U_n$ and represents an open neighbourhood of $X$. Therefore, $U_n$ is open. 

To show that $U_n$ is dense, consider an arbitrary simplicial complex $X\in \Omega$ and an arbitrary $\epsilon >0$. Pick $m>\max\{ \ln(\epsilon^{-1}), n\}$ and 
find a complex $Y\in U_n$ satisfying $X_m=Y_m$. 
This shows that $Y\in U_n$ and $d(X,Y)<\epsilon$, i.e. $U_n$ is dense in $\Omega$. 
\end{proof}

Next we describe a probability measure on $\Omega$. For a subcomplex $Y\subset \Delta_n$ define
$$Z(Y,n)=\{X\in \Omega; X\cap \Delta_n = Y\}.$$
The sets $Z(Y, n)$, with various $(Y, n)$, form a semi-ring $\mathcal A$, see \cite{Kl}, 
and we denote by $\mathcal A'$ the $\sigma$-algebra generated by $\mathcal A'$. An additive measure 
$\mu$ on $\mathcal A$ can be defined by 
\begin{eqnarray}\label{gibbs2}
\mu(Z(Y, n)) = 2^{-H(Y, n)} \quad \mbox{where}\quad  H(Y, n) = |F(Y)|+|E(Y|\Delta_n)|,
\end{eqnarray}
compare with (\ref{gibbs}). Here $E(Y|\Delta_n)$ denotes the set of all simplexes $\sigma\in \Delta_n$ which are external to $Y$, i.e. $\sigma\notin Y$ but $\partial \sigma \in Y$. This is a special case of the measure discussed in \S \S 6, 7 of \cite{Rado}. Theorem 1.53 from \cite{Kl} implies that $\mu$ extends to a probability measure on the $\sigma$-algebra $\mathcal A'$ generated by $\mathcal A$.

\begin{theorem}
The set $\mathcal R\subset \Omega$ belongs to the $\sigma$-algebra $\mathcal A'$ and has full measure, i.e. $\mu(\mathcal R)=1$. 
\end{theorem}
\begin{proof} For a finite subset $U\subset \N$ and for a simplicial subcomplex $A\subset \Delta_U$ of the simplex $\Delta_U$ consider the set 
\begin{eqnarray}
\Omega^{U, L}= \{X \in \Omega; X_U=L\}.
\end{eqnarray}
This set belongs to the $\sigma$-algebra $\mathcal A'$ and has positive measure given by (\ref{gibbs2}). 
Consider also the subset 
$\Omega^{U, L, A, v}\subset \Omega^{U, L}$
consisting of all subcomplexes $X\in \Omega$ satisfying $X_{U\cup v}=L\cup vA$. Here $A\subset L$ is a subcomplex and $v\in \N-U$. 
The conditional probability equals
$$\mu(\Omega^{U, L, A, v}|\Omega^{U, L}) = 2^{-|F(A)|-|E(A|L)|-1}>0,$$
as follows from (\ref{gibbs2}). 
Note that the events $\Omega^{U, L, A, v}$, conditioned on $\Omega^{U, L}$, for various $v$, are independent and the sum of their probabilities is $\infty$. We may therefore apply the Borel-Cantelli Lemma (see \cite{Kl}, p. 51) to conclude that the set of complexes $X\in \Omega^{U, L}$ such that 
$X_{U\cup v}=L\cup vA$ for infinitely many vertexes $v$ has full measure in $\Omega^{U, L}$. 

By taking a finite intersection with respect to all possible subcomplexes $A\subset L$ this implies that the set 
$\Omega_\ast^{U, L}\subset \Omega^{U, L}$ of simplicial complexes $X\in \Omega^{U, L}$ such that for any subcomplex $A\subset L$ there exists infinitely many 
vertexes $v$ with $X_{U\cup v}=L\cup vA$ has full measure in $\Omega^{U, L}$. 
Since $\Omega = \cap_U \cup_{L\subset \Delta_U} \Omega^{U, L}$ (where $U\subset \N$ runs over all finite subsets) we obtain that the set 
$\cap_U \cup_{L\subset \Delta_U} \Omega^{U, L}_\ast$ has measure 1 in $\Omega$. But the latter set 
$\cap_U \cup_{L\subset \Delta_U} \Omega^{U, L}_\ast$ is exactly the set of all Rado complexes $\mathcal R$. 
\end{proof}

\section{Geometric realisation of the Rado complex}\label{secreal}

For a simplicial complex $X$, {\it the geometric realisation} $|X|$ is the set of all functions $\alpha: V(X)\to [0,1]$ such that 
the support ${\rm {supp}}(\alpha)=\{v; \alpha(v)\not=0\}$ is a simplex of $X$ (and hence finite) and $\sum_{v\in X} \alpha(v)=1$, see \cite{Sp}. For a simplex $\sigma\in F(X)$ the symbol $|\sigma|$ denotes the set of all $\alpha\in |X|$ with ${\rm {supp}}(\alpha)\subset \sigma$. The set $|\sigma|$ has natural topology and is homeomorphic to the affine simplex in an Euclidean space. 

{\it The weak topology on} the geometric realisation $|X|$ has as open sets the subsets $U\subset |X|$ such that $U\cap |\sigma|$
is open in $ |\sigma|$ for any simplex $\sigma$.

\begin{theorem}\label{simplex} The Rado complex is isomorphic to a triangulation of the simplex $\Delta_\N$. In particular, 
the geometric realisation $|X|$ of the Rado complex is homeomorphic to the geometric realisation of the 
infinite dimensional simplex $|\Delta_\N|$. 
\end{theorem}

The result of Theorem \ref{simplex} is also stated in preprint \cite{BTT}. 

The proof of Theorem \ref{simplex} given in \cite{Rado} uses the following Lemma:

\begin{lemma}\label{lm20}
Let $X$ be a Rado complex. Then there exists a sequence of finite subsets $U_0\subset U_1\subset U_2\subset \dots\subset V(X)$ such that $\cup U_n = V(X)$ and for any $n=0, 1, 2, \dots$ the induced simplicial complex $X_{U_n}$ is isomorphic to a triangulation $L_n$ of the standard simplex $\Delta_{n+1}$ of dimension $n$.  Moreover, for any $n$ the complex $L_n$ is naturally an induced subcomplex of $L_{n+1}$ and the isomorphisms $f_n: X_{U_n} \to L_n$ satisfy $f_{n+1}|X_{U_n} = f_n$. 
\end{lemma}

Note that the geometric realisation $|X|$ of a Rado complex $X$ (equipped with the weak topology)  does not satisfy the first axiom of countability and hence is not metrizable. This follows from the fact that $X$ is not locally finite. See \cite{Sp}, Theorem 3.2.8. 

The geometric realisation of a simplicial complex carries yet another natural topology, {\it the metric topology}, see \cite{Sp}, p. 111. 
The geometric realisation of $X$ with the metric topology is denoted $|X|_d$. While for finite simplicial complexes the spaces $|X|$ and 
$|X|_d$ are homeomorphic, it is not true for infinite complexes in general. For the Rado complex $X$ the spaces $|X|$ and $|X|_d$ are not homeomorphic. Moreover,  in general, the metric topology is not invariant under subdivisions, see \cite{Mine}, where this issue is discussed in detail. 

The Urysohn metric space \cite{V} is a well-known universal mathematical object;
it is intriguing to examine its relationship to the Rado simplicial complex. The Urysohn universal metric space $U$ is characterised (uniquely, up to isometry) by the following properties: (1) $U$ is complete and separable; (2) $U$ contains an isometric copy of every separable metric space; (3) every isometry between two finite subsets of $U$ can be extended to an isometry of $U$ onto itself. This looks similar to the characterisation of the Rado complex given by Theorem \ref{thm73}. 

V. Uspenskij \cite{U} proved that $\U$ is homeomorphic to the Hilbert space $\ell^2$. 

One may ask whether there exists a natural metric on the Rado complex $X$ turning it into a model for the Urysohn metric space $\U$? 
As a hint we may offer the following observation. The set of vertexes $V(X)$ of $X$ carries the following metric $\delta$: for $x, y\in V(X)$ with $x\not= y$ one sets $\delta(x, y)=1$ iff $x$ and $y$ are connected by an edge; otherwise\footnote{We remind the reader that in the Rado complex $X$ any two vertexes have a common neighbour, i.e. from any vertex $A$ one can get to any vertex $B$ jumping accross one or two edges.}  $\delta(x, y)=2$.
The obtained metric space $(V(X), \delta)$ is
an analogue of the Urysohn universal metric space restricted to countable metric spaces with distance functions taking 
values $1$ and $2$ only. Such metric spaces are in 1-1 correspondence with countable graphs, and our observation follows from the universality of the Rado graph, which is the 1-dimensional skeleton of the Rado complex $X$. 

\vskip 1cm
The author states that there is no conflict of interest. 

\section*{Appendix: \\
On the connectivity of conic complexes}
\begin{center}
by Jonathan Ariel Barmak
\footnote{Universidad de Buenos Aires. Facultad de Ciencias Exactas y Naturales. Departamento de Matem\'atica. Buenos Aires, Argentina. jbarmak@dm.uba.ar}
\end{center}
\vskip 0.5cm

Let $r \in \Z$. A simplicial complex $K$ is said to be $r$-conic if every subcomplex $L\leqslant K$ with at most $r$ vertices is contained in a simplicial cone, or, equivalently, in the closed star $\st_K (v)$ of a vertex $v\in K$. Note that the notions of $0$-conicity and $1$-conicity coincide and that they are equivalent to the complex being non-empty. Every complex is $r$-conic if $r\le -1$.

It was proved in \cite[Theorem 12]{Bar1}, that if a complex $K$ is $5^n$-conic, then it is $n$-connected. The argument uses a sequence $(L_k)_{k\in \N}$ of triangulations of $S^n$ together with a simplicial approximation $\varphi:L_k\to K$ and an idea that allows to reduce the number of vertices in $L_k$ and deform $\varphi$ to a homotopic map. On the other hand a join of $0$-dimensional spheres shows that a $(2n+1)$-conic complex may not be $n$-connected (\cite[Example 6]{Bar1}). The Nerve lemma \cite[Theorem 10.6]{Bjo} can be used to give a simple proof that $(2n+2)$-conicity already implies $n$-connectivity. The proof was given by Kahle in \cite[Theorem 3.1]{Kah} based on a similar result by Meshulam \cite[Proposition 3.1]{Mes}. In fact Meshulam's result deals only with homological connectivity, but gives better bounds under stronger hypotheses. Their results are stated for clique complexes, although the argument holds for general complexes with minor modifications. We state here the result in the general case and give a proof for future reference.

Let $K$ and $L$ be simplicial complexes. We denote by $K \circledast L$ the \textit{non-disjoint join}. It is the simplicial complex whose simplices are the unions $\sigma \cup \tau$ with $\sigma \in K$, $\tau \in L$, and also the simplices of $K$ and the simplices of $L$. If the vertex sets of $K$ and $L$ are disjoint, $K \circledast L$ coincides with the usual join $K *L$. Of course, $\circledast$ is commutative and associative.

\begin{ej}
Let $K$ be a simplicial complex, $L\le K$ a subcomplex and $v$ a vertex of $K$. Then $L \le \st_K (v)$ if and only if the non-disjoint cone $v\circledast L \le K$. Moreover, if $v_1, v_2, \ldots, v_t$ are $t$ vertices of $K$, then $L\le \bigcap \st_K(v_i)$ if and only if $\{v_1,v_2,\ldots, v_t\} \circledast L \le K$, where $\{v_1,v_2,\ldots, v_t\}$ denotes the discrete subcomplex.
\end{ej}

\begin{lema} \label{lema}
Let $K$ be an $r$-conic simplicial complex. Let $t\in \N$ and $v_1, v_2, \ldots, v_t\in K$. Then $S=\bigcap \st_K(v_i)\le K$ is $(r-t)$-conic. In particular, if $t\le r$, $S$ is non-empty.
\end{lema}
\begin{proof}
Let $L\le S$ be a subcomplex of at most $r-t$ vertices. Then $\{v_1,v_2,\ldots, v_t\} \circledast L \le K$ has at most $r$ vertices and thus there exists $v\in K$ such that $\{v_1,v_2,\ldots, v_t\} \circledast L \in \st_K(v)$. In other words $v\circledast \{v_1,v_2,\ldots, v_t\} \circledast L \le K$. Then $v \circledast L \le S$. This means that $L \le \st_S(v)$.
\end{proof}

We recall the statement of the Nerve lemma \cite[Theorem 10.6]{Bjo}.

\begin{teo}[Nerve lemma] 
Let $K$ be a simplicial complex and $\{L_i\}_{i\in I}$ a family of subcomplexes covering $K$. Let $n\ge 0$. If each non-empty intersection $L_{i_1}\cap L_{i_2} \cap \ldots \cap L_{i_t}$ is $(n-t+1)$-connected for every $1\le t\le n+1$, then $K$ is $n$-connected if and only if the nerve $\mathcal{N}(\{L_i\}_{i\in I})$ is $n$-connected.
\end{teo}    

\begin{teo}
Let $n\ge 0$. If a simplicial complex $K$ is $(2n+2)$-conic, then it is $n$-connected. 
\end{teo}   
\begin{proof}
Let $\mathcal{U}=\{\st_K(v)\}_{v\in K}$. We claim that $\mathcal{N}(\mathcal{U})$ has complete $(2n+1)$-skeleton, that is any set of at most $2n+2$ vertices is a simplex. Indeed, the intersection of the closed stars of $t\le 2n+2$ vertices of $K$ is non-empty by Lemma \ref{lema}. In particular $\mathcal{N}(\mathcal{U})$ is $2n$-connected, so it is certainly $n$-connected.

By the Nerve lemma, it suffices to verify that for each $1\le t\le n+1$ and vertices $v_1,v_2,\ldots, v_t$, $S=\bigcap \st_K(v_i)$ is $(n-t+1)$-connected. For $t=1$ this is trivial. For $2\le t\le n+1$, $0\le n-t+1<n$ and by induction it suffices to check that $S$ is $(2n-2t+4)$-conic. But Lemma \ref{lema} says that $S$ is $(2n+2-t)$-conic, and since $t\ge 2$, $2n+2-t\ge 2n-2t+4$.
\end{proof}

\end{document}